\documentclass[11pt]{article}

\usepackage{amsmath}
\usepackage{amssymb}
\usepackage{amscd}
\usepackage{amsthm}
\usepackage{indentfirst}

\newcommand{\Z}{\ensuremath{\mathbb{Z}}}
\newcommand{\Q}{\ensuremath{\mathbb{Q}}}

\newcommand{\F}{\ensuremath{\mathbb{F}}}

\usepackage[margin=2.6cm]{geometry}

\theoremstyle{plain}
\newtheorem{theorem}{Theorem}[section]
\newtheorem{lemma}[theorem]{Lemma}
\newtheorem{proposition}[theorem]{Proposition}

\newtheorem{conjecture}[theorem]{Conjecture}

\theoremstyle{definition}
\newtheorem{definition}[theorem]{Definition}
\newtheorem{remark}[theorem]{Remark}

\DeclareMathOperator{\Spec}{Spec}

\DeclareMathOperator{\Pic}{Pic}
\DeclareMathOperator{\Cl}{Cl}

\DeclareMathOperator{\rk}{rk}
\DeclareMathOperator{\car}{char}

\newcommand{\PP}{\mathbb{P}}
\newcommand{\E}{\mathcal{E}}
\newcommand{\NN}{\mathcal{N}}

\newcommand{\OO}{\mathcal{O}}

\begin{document}

\title{Elliptic surfaces over $\PP^1$ and large class groups of number fields}

\author{Jean Gillibert \and Aaron Levin\thanks{Supported in part by NSF grant DMS-1352407.}
}

\date{May 2019}

\maketitle

\begin{abstract}
Given a non-isotrivial elliptic curve over $\Q(t)$ with large Mordell-Weil rank, we explain how one can build, for suitable small primes $p$, infinitely many fields of degree $p^2-1$ whose ideal class group has a large $p$-torsion subgroup. As an example, we show the existence of infinitely many cubic fields whose ideal class group contains a subgroup isomorphic to $(\Z/2\Z)^{11}$.
\end{abstract}

%\keywords{Ideal class groups; elliptic surfaces.}

%\ccode{Mathematics Subject Classification 2010: 11R29, 11G05, 14J27}

%%%%%%%%%%%%%%%%%%%%%%%%%%%%%%%%%%%%%%%%%%%%%

\section{Introduction}

%%%%%%%%%%%%%%%%%%%%%%%%%%%%%%%%%%%%%%%%%%%%%

Our initial motivation for the present paper is the following conjecture on class groups of number fields, which belongs to folklore, and is a consequence of the Cohen-Lenstra heuristics. If $K$ is a number field, we denote by $\OO_K$ its ring of integers, and by $\Cl(K)$ its ideal class group.

\begin{conjecture}
\label{conj1}
Let $k$ be a number field, let $p$ be a prime number, and let $d>1$ be an integer. Then $\dim_{\F_p} \Cl(L)[p]$ is unbounded when $L/k$ runs through all extensions of degree $[L:k]=d$.
\end{conjecture}

When $d=p$, and more generally when $p$ divides $d$, this conjecture follows from class field theory \cite{brumer65,RZ} (or see \cite[Th.~3.4]{levin07} for a proof aligned with the techniques of the present paper).  On the other hand, when $p$ and $d$ are coprime,
there is not a single case where Conjecture~\ref{conj1} is known to hold. For example, given a prime $p$, it is known \cite{Yam} that there exist infinitely many imaginary quadratic fields $L/\Q$ such that $\dim_{\F_p}\Cl(L)[p]\geq 2$. For $p\geq 7$ this is currently the best known result on $p$-ranks of class groups of quadratic fields.

In the present paper, we investigate this conjecture in the specific case when $d=p^2-1$. Our strategy is closely related to the techniques developed in \cite{gl12} and \cite{bg16}. The main new ingredient is the following result of \cite{gl18a}: given a non-isotrivial elliptic surface over $k(t)$ with large rank, for almost all primes $p$ one is able to produce a curve $C$ which admits a morphism $C\to \PP^1_k$ of degree $p^2-1$, and whose Picard group has large $p$-rank (see Theorem~\ref{thm_kummer}).  This construction allows one to use large Mordell-Weil rank over $\mathbb{Q}(t)$ to produce number fields of degree $p^2-1$ whose ideal class group has large $p$-rank.  In particular, we show the existence of infinitely many cubic fields whose ideal class group has $2$-rank $r\geq 11$.  This improves on a result of Kulkarni \cite{kulkarni18}, who proved this statement with $r\geq 8$ (see Remark \ref{rmq:2rank}).

%%%%%%%%%%%%%%%%%%%%%%%%%%%%%%%%%%%%%%%%%%%%%

\subsection{The main result}
\label{sec:mainresult}

Let $k$ be a number field, and let $E$ be an elliptic curve over $k(t)$. We denote by $\E\to\PP^1_k$ the N{\'e}ron model of $E$ over $\PP^1_k$, by which we mean the group scheme model, which is the smooth locus of N{\'e}ron's minimal regular model (see \cite[\S{}1.5]{NeronModels} or \cite[\S{}10.2]{liu}).

By abuse of notation, we identify closed points of $\PP^1_k$ and discrete valuations of $k(t)$.
If $v\in \PP^1_k$ is such a point, we denote by $\E_v$ the special fiber of $\E$ at $v$, and by $\Phi_v$ the group of connected components of $\E_v$. The \emph{Tamagawa number} of $E$ at $v$ is by definition the order of $\Phi_v(k_v)$, and we denote it by $c_v$. If $E$ has good reduction at $v$, then $\Phi_v=\{0\}$ and $c_v=1$. The set of places of bad reduction being finite, we have that $c_v=1$ for all but finitely many $v$.

\begin{theorem}
\label{thm:Main}
Let $k$ be a number field and let $E$ be an elliptic curve over $k(t)$. Assume that $E$ does not have universal bad reduction at any prime of $k$ (Definition \ref{dfn:UBR}). Let $p$ be a prime number such that $E[p]$ is irreducible as a Galois module over $\overline{k}(t)$.
\begin{enumerate}
\item[1)] If $p\geq 3$, there exist infinitely many field extensions $L/k$ of degree $p^2-1$ such that
\begin{equation}
\label{classbound1}
\dim_{\F_p} \Cl(L)[p]-\dim_{\F_p} \Cl(k)[p] \geq \rk_\Z E(k(t)) - \#\{v\in \PP^1_k, p|c_v\} - \rk_\Z \mathcal{O}_L^{\times} + \rk_\Z \mathcal{O}_k^{\times}.
\end{equation}
\item[2)] If $p=2$, there exist infinitely many field extensions $L/k$ of degree $3$ such that
\begin{equation}
\label{classbound2}
\begin{split}
\dim_{\F_2} \Cl(L)[2] -\dim_{\F_2} &\Cl(k)[2] \geq \rk_\Z E(k(t)) -\#\{v\in \PP^1_k, 2|c_v\} -\rk_\Z \mathcal{O}_L^{\times} +\rk_\Z \mathcal{O}_k^{\times}\\
&-\#\{v\in \PP^1_k, \text{the red. type of $\E$ at $v$ is $\mathrm{I}_{2n}^*$ for some $n\geq 0$} \}.
\end{split}
\end{equation}
\end{enumerate}
\end{theorem}

The proof is given in \S{}\ref{sec:avoidingUBR}. The assumption that $E$ does not have universal bad reduction at any prime of $k$ is a technical condition detailed in \S{}\ref{sec:avoidingUBR}. Given a Weierstrass equation for the elliptic curve $E$, this condition can be verified numerically.

If $E$ is not isotrivial, then $E[p]$ is irreducible as a Galois module over $\overline{k}(t)$ for all but finitely many primes $p$. If in addition $E$ does not have universal bad reduction at any prime of $k$, then the conclusion of Theorem~\ref{thm:Main} holds for all but finitely many primes $p$.

\begin{remark}
When $p$ gets large, the term $\rk_\Z \mathcal{O}_L^{\times}$ also gets large: according to Dirichlet's unit theorem, it is bounded below by $\frac{p^2-1}{2}-1$. Therefore, given an elliptic surface $E$, Theorem~\ref{thm:Main} should be applied to primes $p$ which are small with respect to the rank of $E$.
\end{remark}

\begin{remark}
When $p=2$, the condition that $E$ does not have universal bad reduction at any prime of $k$ can be replaced by the condition that its N{\'e}ron model $\E$ has a fiber of type $\mathrm{II}$, $\mathrm{II}^*$, $\mathrm{IV}$ or $\mathrm{IV}^*$ at some $v\in \PP^1(k)$. In fact, we obtain in this case a stronger version of \eqref{classbound2}, in which the contribution from the units is removed; see Proposition~\ref{prop2}.
\end{remark}

\begin{remark}
The fields $L/k$ of degree $p^2-1$ in \eqref{classbound1} have in fact a canonical subfield $L'$ with $[L':k]=p+1$. In specific situations, the equation of $E$ being given, one may perform additional computations in order to obtain a lower bound on $\dim_{\F_p} \Cl(L')[p]$. See Remark~\ref{rmq:Aaron} for details.
\end{remark}

%%%%%%%%%%%%%%%%%%%%%%%%%%%%%%%%%%%%%%%%%%%%%

\subsection{Example: cubic fields}

Applying Theorem~\ref{thm:Main} to an elliptic curve over $\Q(t)$ with large rank constructed by Kihara \cite{Kih}, we obtain the following result (see \S{}\ref{sec:Kihara} for the proof).

\begin{theorem}
\label{thm:cubic}
There exists a trigonal curve $C$, defined over $\Q$, such that
$$
\dim_{\F_2} \Pic(C)[2] \geq 12.
$$
Moreover, there exists a degree three morphism $\phi:C\to\PP^1$ such that, for all but $O(\sqrt{N})$ integers $t\in\mathbb{Z}$ with $|t|\leq N$, the field of definition of $\phi^{-1}(t)$ is a cubic field $K_t$ with exactly one real place, satisfying
$$
\dim_{\F_2} \Cl(K_t)[2] \geq 11.
$$
\end{theorem}

\begin{remark}
\label{rmq:2rank}
Nakano \cite{Nakano} proved that there are infinitely many cubic fields $K$ whose ideal class group has $2$-rank at least $6$. Recently, Kulkarni \cite{kulkarni18} has improved on this result, obtaining infinitely many cubic fields $K$ whose ideal class group has $2$-rank at least $8$. To our knowledge, this is the best previously known result on the existence of an infinite family of cubic fields whose class group has large $2$-rank. In fact, our method is closely related to Kulkarni's, in the sense that in both cases one considers trigonal curves which come from the $2$-torsion of an elliptic fibration.
\end{remark}

%%%%%%%%%%%%%%%%%%%%%%%%%%%%%%%%%%%%%%%%%%%%%

\subsection{Picard groups of curves over finite fields}

Picard groups of curves over finite fields are considered natural analogues of class groups of number fields. Thus, we include in this section a remark on the construction of curves over finite fields whose Jacobian contains a large torsion subgroup.

In contradistinction to Conjecture~\ref{conj1}, this geometric analogue does not appear to have been the subject of intense investigation. Nevertheless, as in the number field case, one can make the following remark: the rational $p$-torsion subgroup of Jacobians of curves of gonality $p$ can be made as large as possible by considering equations of the form $y^p=f(x)$, where $f$ splits as a product of linear factors. It seems harder to find curves of gonality coprime to $p$ whose Jacobian contains a rational $p$-torsion subgroup which is large (relative to the gonality). We shall give some results towards this goal.

Ulmer has shown \cite{ulmer02} that, given a prime $q$, one can find non-isotrivial elliptic curves over $\F_q(t)$ with arbitrarily large Mordell-Weil rank. More precisely, given an integer $n\geq 1$, he considers the elliptic curve $E_n$ defined by the equation
\begin{equation}
\tag{$E_n$}
y^2+xy=x^3+t^d, \qquad \text{where $d=q^n+1$},
\end{equation}
and he proves that the rank of $E_n$ over $\F_q(t)$ is at least $(q^n-1)/2n$.

Let us assume that $q\geq 5$. Then the trigonal curve $C_{2,n}\to\PP^1$ describing the $2$-torsion of $E_n$ is defined by the equation
\begin{equation}
\tag{$C_{2,n}$}
4x^3 + x^2 - 4t^d = 0.
\end{equation}

We claim that $C_{2,n}$ is geometrically irreducible. Indeed, if the equation above had a root $x$ in $\overline{\F}_q(t)$, then by Gauss's lemma this root would be a polynomial which divides $t^d$, and this leads to a contradiction.

Applying Theorem~\ref{thm_kummer} to the elliptic curves $E_n$, we obtain in \S{}\ref{sec:ellsurf} the following result.

\begin{theorem}
\label{thm:Ulmer}
Let $q\geq 5$ be a prime. Then
\begin{enumerate}
\item[1)] The family of trigonal curves $C_{2,n}$ defined above (over $\F_q$) satisfies
$$
\dim_{\F_2} \Pic(C_{2,n})[2] \geq \frac{q^n-1}{2n} - 3.
$$
\item[2)] Given an integer $n\geq 1$, for all but finitely many primes $p$, there exists a curve $C$ over $\F_q$ which admits a morphism $C\to \PP^1$ of degree $p^2-1$, such that
$$
\dim_{\F_p} \Pic(C)[p] \geq \frac{q^n-1}{2n} - 2.
$$
\end{enumerate}
\end{theorem}

\begin{remark}
Using similar techniques, it has been proved in \cite[Section~5]{cesnavicius15} that, if $p$ and $q$ are two distinct primes, there exists a constant $m$ depending only on $p$ such that the size of $\Pic(C)[p]$ is unbounded when $C$ runs through hyperelliptic curves over $\F_{q^m}$.
\end{remark}

%%%%%%%%%%%%%%%%%%%%%%%%%%%%%%%%%%%%%%%%%%%%%

%%%%%%%%%%%%%%%%%%%%%%%%%%%%%%%%%%%%%%%%%%%%%

\section{Proofs}

%%%%%%%%%%%%%%%%%%%%%%%%%%%%%%%%%%%%%%%%%%%%%

\subsection{Elliptic surfaces over $\PP^1_k$}
\label{sec:ellsurf}

In this section, we briefly recall the main result of \cite{gl18a} in the setting of elliptic surfaces over $\PP^1_k$. We refer the reader to \emph{loc. cit.} for further details and comments.

Let $k$ be a perfect field of characteristic not $2$ or $3$.  In the applications we have in mind, $k$ may be a number field, or a finite field, or the algebraic closure of such fields.

Let $E$ be an elliptic curve over $k(t)$, and let $\E\to\PP^1_k$ be its N{\'e}ron model. If $v$ is a closed point of $\PP^1_k$, we let $\E_v$, $\Phi_v$, and the Tamagawa number $c_v$ be as in \S{}\ref{sec:mainresult}.  We denote by $\Phi$ the skyscraper sheaf over $\PP^1_k$ whose fiber at $v$ is $\Phi_v$, and by $\E^{p\Phi}$ the inverse image of $p\Phi$ by the natural map $\E\to\Phi$.

Given a prime number $p\neq \car(k)$, we denote by $\mathcal{E}[p]\to \PP^1_k$ the group scheme of $p$-torsion points of $\E$. Because $p\neq \car(k)$ the multiplication-by-$p$ map on $\mathcal{E}$, which is a smooth group scheme over $\PP^1_k$, is {\'e}tale, and in particular $\mathcal{E}[p]\to \PP^1_k$ is {\'e}tale. On the other hand, $\mathcal{E}[p]\to \PP^1_k$ is quasi-finite, but not finite in general. This is why we consider the smooth compactification of $\E[p]\setminus \{0\}$, that we denote by $C$, endowed with its canonical finite map $C\to \PP^1_k$ of degree $p^2-1$. In fact, one can show that $\E[p]\setminus \{0\}$ coincides with the {\'e}tale locus of $C\to \PP^1_k$; we refer to \cite[Proof of Lemma~2.8]{gl18a} for further details.

The following result is a special case of \cite[Theorem~1.1]{gl18a}. Assuming that $C$ is geometrically integral, it provides an upper bound on the rank of $E$ in terms of the Tamagawa numbers $c_v$ and the $p$-torsion in the Picard group of $C$. Its proof relies on $p$-descent techniques, analogous to the number field case.

\begin{theorem}
\label{thm_kummer}
Let $p\neq \car(k)$ be a prime number, and let $C$ be the smooth compactification of $\E[p]\setminus \{0\}$. Assume that $C$ is geometrically integral, or equivalently that $E[p]$ is irreducible as a Galois module over $\overline{k}(t)$. Then
\begin{enumerate}
\item[1)] There is an injective morphism
\begin{equation}
\label{eq:descent_map}
\E^{p\Phi}(\PP^1_k)/p\E(\PP^1_k) \longrightarrow  \Pic(C)[p].
\end{equation}
\item[2)] If $p\geq 3$, we have
\begin{equation}
\label{inequality1}
\begin{split}
\dim_{\F_p} \Pic(C)[p] &\geq \dim_{\F_p} \E^{p\Phi}(\PP^1_k)/p\E(\PP^1_k) \\
&\geq \rk_{\Z} E(k(t)) - \#\{v\in \PP^1_k, p\mid c_v\},
\end{split}
\end{equation}
where $c_v$ denotes the Tamagawa number of $\E$ at $v$.
\item[3)] If $p=2$, then
\begin{equation}
\label{inequality2}
\begin{split}
\dim_{\F_2} \Pic(C)[2] \geq & \dim_{\F_2} \E^{2\Phi}(\PP^1_k)/2\E(\PP^1_k) \\
\geq & \rk_{\Z} E(k(t)) - \#\{v\in \PP^1_k, 2\mid c_v\} \\
& -\#\{v\in \PP^1_k, \text{the red. type of $\E$ at $v$ is $\mathrm{I}_{2n}^*$ for some $n\geq 0$} \}.
\end{split}
\end{equation}
\end{enumerate}
\end{theorem}

In fact, the injective morphism \eqref{eq:descent_map} is obtained by composing maps
\begin{equation}
\label{eq:H1map}
\begin{CD}
\E^{p\Phi}(\PP^1_k)/p\E(\PP^1_k) @>>> H^1(C,\mu_p) @>>> \Pic(C)[p]
\end{CD}
\end{equation}
the first being obtained by Kummer theory on $\E$, and the second being the natural one. The cohomology group in the middle is computed with respect to the {\'e}tale topology, or equivalently the fppf topology since $\mu_p$ is {\'e}tale over $C$ under the assumptions of Theorem~\ref{thm_kummer}.

A comment on the terminology: when we say that $E$ has a fiber of type $\mathrm{I}_{2n}^*$ at $v$, we mean it over $k_v$, and not just over $\overline{k}$. More precisely, this means that the Kodaira type of $\E_v$ over $\overline{k}$ is $\mathrm{I}_{2n}^*$, and that the four components of $\E_v$ are rational over $k_v$, in other terms $\Phi_v(k_v)\simeq (\Z/2)^2$. In general, the reduction type at $v$ can be described by the data of the reduction type over $\overline{k}$ together with the action of the absolute Galois group of $k_v$ on $\Phi_v$. See Liu's book \cite[\S{}10.2]{liu}.

\begin{remark}
\label{rmq:Aaron}
The curve $C$ in the statement of Theorem~\ref{thm_kummer} corresponds to the field over which $E$ acquires one rational $p$-torsion point. One can also introduce a curve $C'$ corresponding to the field over which $E$ has a rational cyclic subgroup of order $p$; then we have canonical maps
$$
C \longrightarrow C' \longrightarrow \PP^1_k
$$
of degree $p-1$ and $p+1$, respectively. Given a specific example of a curve $E$, one can find equations for $C$ and $C'$, and compute the kernel of the norm map $\Pic(C)[p]\to \Pic(C')[p]$. If this kernel is small enough then, using techniques of \S{}\ref{sec:avoidingUBR}, one could use the map $C'\to \PP^1_k$ of degree $p+1$ in order to build extensions $L'/k$ of degree $p+1$ with $\dim_{\F_p} \Cl(L')[p]$ large.
\end{remark}

\begin{proof}[Proof of Theorem~\ref{thm:Ulmer}]
Ulmer has checked that $E_n$ has reduction $I_1$ at places dividing $(1-2^43^2t^d)$, and has split multiplicative reduction $I_d$ at $t=0$. The last place of bad reduction is $t=\infty$, where the possible reduction types are $\mathrm{I}_1$, $\mathrm{II}$, $\mathrm{II}^*$, $\mathrm{IV}$, $\mathrm{IV}^*$ or $\mathrm{I}_0^*$ depending on $q^n+1\pmod{6}$. Then 1) is a consequence of \eqref{inequality2} in Theorem~\ref{thm_kummer}, the error term $-3$ being obtained as $-1-2$, where $-1$ corresponds to the fiber at $t=0$ whose Tamagawa number is divisible by two, and $-2$ corresponds to the worst case when the reduction type at infinity is $\mathrm{I}_0^*$. The second statement follows similarly from \eqref{inequality1} in Theorem~\ref{thm_kummer}, combined with the following observation: the elliptic curve $E_n$ is not isotrivial, hence according to the geometric version of Shafarevich's theorem, $E_n[p]$ is $\overline{k}$-irreducible for all but finitely many primes $p$.
\end{proof}

%%%%%%%%%%%%%%%%%%%%%%%%%%%%%%%%%%%%%%%%%%%%%

\subsection{Fibers of type $\mathrm{II}$, $\mathrm{II}^*$, $\mathrm{IV}$, $\mathrm{IV}^*$}

Let $k$ be a number field, and let $E$ be an elliptic curve over $k(t)$ defined by a Weierstrass equation
$$
y^2 = x^3 + a(t) x + b(t).
$$

Let us assume that $E(\overline{k}(t))[2]=\{0\}$, which is equivalent to saying that $E[2]$ is irreducible as a Galois module over $\overline{k}(t)$. Then the smooth compactification of $\E[2]\setminus\{0\}$ is none other than the smooth, projective,  geometrically integral curve $C$ defined by the affine equation
$$
x^3 + a(t) x + b(t) = 0,
$$
and the canonical map $C\to \PP^1_k$ is just the $t$-coordinate map, which has degree $3$.

\begin{lemma}
\label{total_ramif}
Assume $\mathcal{E}\to \PP^1_k$ has a fiber of type $\mathrm{II}$, $\mathrm{II}^*$, $\mathrm{IV}$ or $\mathrm{IV}^*$ at some $v\in \PP^1_k$. Then the $t$-coordinate map $C\to \PP^1_k$ is totally ramified above $v$.
\end{lemma}

\begin{proof}
This follows from the proof of Lemma~2.8 in \cite{gl18a}.
\end{proof}

\begin{proposition}
\label{prop2}
Let $E$ be an elliptic curve over $k(t)$ such that $E(\overline{k}(t))[2]=\{0\}$. Assume in addition that its N{\'e}ron model has a fiber of type $\mathrm{II}$, $\mathrm{II}^*$, $\mathrm{IV}$ or $\mathrm{IV}^*$ at some $v\in \PP^1(k)$. Then there exist infinitely many cubic field extensions $L/k$ such that
\begin{equation}
\begin{split}
\dim_{\F_2} \Cl(L)[2] - \dim_{\F_2} \Cl(k)[2]  &\geq \rk_{\Z} E(\Q(t)) - \#\{v\in \PP^1_k, 2|c_v\} \\
& -\#\{v\in \PP^1_k, \text{the red. type of $\E$ at $v$ is $\mathrm{I}_{2n}^*$ for some $n\geq 0$} \}.
\end{split}
\end{equation}
\end{proposition}

\begin{proof}
Let $C$ be the trigonal curve defining $2$-torsion points as above. According to Lemma~\ref{total_ramif}, the natural trigonal map $C\to \PP^1_k$ is totally ramified over some rational point. According to \cite[Theorem~1.4]{bg16}, there exist infinitely many number fields $L/k$ of degree $3$ such that
$$
\dim_{\F_2} \Cl(L)[2] - \dim_{\F_2} \Cl(k)[2] \geq \dim_{\F_2} \Pic(C)[2].
$$
The result then follows from the last statement of Theorem~\ref{thm_kummer}.
\end{proof}

%%%%%%%%%%%%%%%%%%%%%%%%%%%%%%%%%%%%%%%%%%%%%

\subsection{Avoiding universal bad reduction}
\label{sec:avoidingUBR}

We now examine proving inequalities as in Proposition~\ref{prop2} in the absence of fibers of type $\mathrm{II}$, $\mathrm{II}^*$, $\mathrm{IV}$ or $\mathrm{IV}^*$, and working with any prime number $p$.

For that purpose, we first introduce the notion of \emph{universal bad reduction} for an elliptic family defined over a number field.

Let $k$ be a number field, and let $E$ be an elliptic curve over $k(t)$ defined by a Weierstrass equation
$$
y^2=x^3+a(t)x+b(t)
$$
with $a(t),b(t)\in \OO_k[t]$.  Let $\E\to \PP^1_k$ be the N{\'e}ron model of $E$, and let $\Sigma\subset\PP^1_k$ be the set of places of bad reduction of $\E$.

\begin{definition}
\label{dfn:UBR}
Let $\mathfrak{p}$ be a prime ideal in $\OO_k$.  We will say that $\mathfrak{p}$ is a prime of \emph{universal bad reduction} for $E$ if $\mathfrak{p}$ is a prime of bad reduction for every elliptic fiber $\E_t$, $t\in \PP^1(k)\setminus \Sigma$.  
\end{definition}

\begin{remark}
Let $\Delta(t):=-16(4a(t)^3+27b(t)^2)$ be the discriminant of the Weierstrass equation defining $E$.  If $\mathfrak{p}$ is a prime of universal bad reduction for $E$, then $\Delta(t) \pmod{\mathfrak{p}}$ is divisible by $t^q-t$, where $q=N(\mathfrak{p})$ is the (absolute) norm of $\mathfrak{p}$.  In particular, if $\Delta(t) \pmod{\mathfrak{p}}$ is not identically zero, then we must have $N(\mathfrak{p})\leq \deg \Delta$.
\end{remark}

Before we state our main result, we set some notation: if $R$ is a commutative ring, we denote by $H^1(R,\mu_p)$ the fppf cohomology group $H^1(\Spec(R),\mu_p)$. If $K$ is a number field or a local field, then we identify $H^1(\mathcal{O}_K,\mu_p)$ with a subgroup of $H^1(K,\mu_p)$ via the natural restriction map.

\begin{theorem}
\label{thmclassgp}
Suppose that $E$ does not have universal bad reduction at any prime of $k$. Assume that $E[p]$ is an irreducible Galois module over $\overline{k}(t)$. Let $C$ be the smooth compactification of $\E[p]\setminus \{0\}$, and let $H$ denote the image of the injective map from \eqref{eq:H1map}
$$
\E^{p\Phi}(\PP^1_k)/p\E(\PP^1_k) \hookrightarrow H^1(C,\mu_p).
$$
Then there exists a map $\psi:C \to \PP^1_k$ of degree $p^2-1$ such that, for all but $O(\sqrt{N})$ integers $t\in\mathbb{Z}$ with $|t|\leq N$, we have that:
\begin{enumerate}
\item[1)] $P_{t,\psi}:=\psi^{-1}(t)$ is the spectrum of a field $k(P_{t,\psi})$, with $[k(P_{t,\psi}):k]=p^2-1$;
\item[2)] the image of $H$ under the specialization map $P_{t,\psi}^*:H^1(C,\mu_p)\to H^1(k(P_{t,\psi}),\mu_p)$ lands into the subgroup $H^1(\mathcal{O}_{k(P_{t,\psi})},\mu_p)$;
\item[3)] the specialization map $P_{t,\psi}^*$ above is injective on $H$.
\end{enumerate}
\end{theorem}

\begin{proof}
Let $X_1,\dots,X_r$ be independent $\mu_p$-torsors over $C$ generating $H$. Then by the Chevalley-Weil theorem, there exists a finite set $T$ of places of $k$ such that the $X_i$ can be extended to $\mu_p$-torsors in $H^1(\mathcal{C}_T,\mu_p)$, where $\mathcal{C}_T$ denotes a smooth projective model of $C$ over $\Spec(\mathcal{O}_{k,T})$. In particular, $T$ contains the set of bad places of $C$.

We denote by $\phi:C\to\PP^1_k$ the natural map $(x,t)\mapsto t$, of degree $p^2-1$. For each $t\in \PP^1(k)$, we let $P_t:=\psi^{-1}(t)$. It follows from the construction of $T$ and the projectivity of $\mathcal{C}_T$ that, for each $t\in \PP^1(k)$, the image of $H$ under the specialization map $P_t^*:H^1(C,\mu_p)\to H^1(k(P_t),\mu_p)$ lands into the subgroup $H^1(\mathcal{O}_{k(P_t),T},\mu_p)$.

Let us now pick $\mathfrak{p}\in T$. By assumption, $E$ does not have universal bad reduction at $\mathfrak{p}$, and hence there exists $t_\mathfrak{p}\in\PP^1(k)$ such that $E_{t_\mathfrak{p}}$ has good reduction at $\mathfrak{p}$.

Let $\NN_{t_\mathfrak{p}}$ be the N{\'e}ron model of $E_{t_\mathfrak{p}}$ over $\Spec(\mathcal{O}_{k,\mathfrak{p}})$. Then $\NN_{t_\mathfrak{p}}$ is an abelian scheme and it follows that the map $[p]:\NN_{t_\mathfrak{p}}\to \NN_{t_\mathfrak{p}}$ is an epimorphism for the fppf topology, regardless of the residue characteristic of $\mathfrak{p}$.

Let $\mathfrak{q}|\mathfrak{p}$ be a place of $k(P_{t_\mathfrak{p}})$ above $\mathfrak{p}$, and let $P_\mathfrak{q}\in C(k(P_{t_\mathfrak{p}})_\mathfrak{q})$ be the corresponding localization of $P_{t_\mathfrak{p}}$. More geometrically, $P_{t_\mathfrak{p}}\otimes k_\mathfrak{p}$ is the disjoint union of the $P_\mathfrak{q}$, hence $(E_{t_\mathfrak{p}}[p]\setminus\{0\})\otimes k_\mathfrak{p}$ is the union of the $P_\mathfrak{q}$. The Weil pairing induces a map
$$
w:\NN_{t_\mathfrak{p}}[p]\longrightarrow \prod_{\mathfrak{q}|\mathfrak{p}} \mathrm{Res}_{\mathcal{O}_{k(P_t),\mathfrak{q}}/\mathcal{O}_{k,\mathfrak{p}}}\mu_p.
$$
We have a commutative diagram
$$
\begin{CD}
E_{t_\mathfrak{p}}(k_\mathfrak{p})/p @= \NN_{t_\mathfrak{p}}(\mathcal{O}_{k,\mathfrak{p}})/p @>>> \prod_{\mathfrak{q}|\mathfrak{p}} H^1(\mathcal{O}_{k(P_t),\mathfrak{q}},\mu_p) \\
@At_\mathfrak{p}^*AA @. @VVV \\
\E^{p\Phi}(\PP^1_k)/p\E(\PP^1_k) @>>> H^1(C,\mu_p) @>(P_{t_\mathfrak{p}}\otimes k_\mathfrak{p})^*>> \prod_{\mathfrak{q}|\mathfrak{p}} H^1(k(P_t)_\mathfrak{q},\mu_p) \\
\end{CD}
$$
in which the upper right map comes from fppf Kummer theory on the abelian scheme $\NN_{t_\mathfrak{p}}$, combined with the Weil pairing map. The vertical left map is obtained by specializing the elliptic family at $t_\mathfrak{p}\in \PP^1(k_\mathfrak{p})$.

This proves that, under the map $P_{t_\mathfrak{p}}^*$, all elements of $H$ (and, in fact, all $\mu_p$-torsors built from Kummer theory on $E$) specialize to torsors which are integral at all places of $k(P_t)$ above $\mathfrak{p}$. By a variant of Krasner's lemma, similar to \cite[Lemma~2.3]{bg16}, the same holds for any $t\in\PP^1(k)$ which is $\mathfrak{p}$-adically close enough to $t_\mathfrak{p}$.

Finally, let us recall that an element of $H^1(\mathcal{O}_{k(P_t),T},\mu_p)$ belongs to $H^1(\mathcal{O}_{k(P_t)},\mu_p)$ if and only if it belongs to $H^1(\mathcal{O}_{k(P_t),\mathfrak{q}},\mu_p)$ for each prime $\mathfrak{q}\in T_{k(P_t)}$. Putting everything together, we may conclude that, if $t\in\PP^1(k)$ is $\mathfrak{p}$-adically close enough to $t_\mathfrak{p}$ for every $\mathfrak{p}$ in $T$, then the image of $H$ under $P_t^*$ lands into $H^1(\mathcal{O}_{k(P_t)},\mu_p)$.

Let $T_\Z$ denote the set of prime numbers lying below primes in $T$, and let $\psi:C \to \PP^1_k$ be the map defined by
$$
\psi:=\frac{1}{(\prod_{p\in T_\Z}p)^M}(\phi-t_0),
$$
where $t_0\in k$ is $\mathfrak{p}$-adically close enough to $t_\mathfrak{p}$ for every $\mathfrak{p}\in T$, and $M$ is a sufficiently large positive integer.
The conclusion of the theorem follows from a quantitative version of Hilbert's irreducibility Theorem (see \cite[Theorem~2.1]{gl12}) applied to the composite map
\begin{equation}
\label{HITcover}
\begin{CD}
X_1\times_C\dots\times_C X_r @>>> C @>\psi>> \PP^1_k.\\
\end{CD}
\end{equation}

In this final step we implicitly use the fact that $X_1\times_C\dots\times_C X_r$ is geometrically irreducible, which can be proved as follows: $H$ injects into $\Pic(C)$ according to the first statement of Theorem~\ref{thm_kummer}, and this remains true over $\overline{k}$ by injectivity of the map $\Pic(C)\rightarrow \Pic(C\times_k \overline{k})$. Hence $H$ injects into $H^1(C\times_k \overline{k},\mu_p)$ via the natural base-change map; in other terms, $X_1,\dots,X_r$ remain independent over $\overline{k}$.
\end{proof}

\begin{proof}[Proof of Theorem~\ref{thm:Main}]
Given the inequalities \eqref{inequality1} and \eqref{inequality2} from Theorem~\ref{thm_kummer}, the result follows from Theorem~\ref{thmclassgp} combined with the following fact: when applying Hilbert's irreducibility theorem to the cover \eqref{HITcover}, it is always possible to ensure that the field extension obtained is linearly disjoint from a given extension fixed in advance. If one fixes the latter extension to be the compositum of the fields that correspond to torsors in $H^1(\mathcal{O}_k,\mu_p)$, then the natural map $H^1(\mathcal{O}_k,\mu_p)\to H^1(\mathcal{O}_{k(P_{t,\psi})},\mu_p)$ is injective, and it follows from Kummer theory (see the proof of Theorem~1.3 in \cite{bg16}) that
\begin{equation*}
\begin{split}
\dim_{\F_p} H^1(\mathcal{O}_{k(P_{t,\psi})},\mu_p)/H^1(\mathcal{O}_k,\mu_p) = &\dim_{\F_p} \Cl(k(P_{t,\psi}))[p] -\dim_{\F_p} \Cl(k)[p]  \\
& + \rk_\Z \mathcal{O}_{k(P_{t,\psi})}^{\times}- \rk_\Z \mathcal{O}_k^{\times}.
\end{split}
\end{equation*}
To conclude, it suffices to point out that the image of the subgroup $H$ from Theorem~\ref{thmclassgp} by the map $P_{t,\psi}^*$ injects into the quotient group on the left.
\end{proof}

%%%%%%%%%%%%%%%%%%%%%%%%%%%%%%%%%%%%%%%%%%%%%

\subsection{Proof of Theorem~\ref{thm:cubic}}
\label{sec:Kihara}

Let $E$ be an elliptic curve over $\mathbb{Q}(t)$ without a rational $2$-torsion point over $\overline{\mathbb{Q}}(t)$.  We may assume that $E$ is defined by an equation
\begin{equation}
\label{QtEq}
y^2=x^3+a(t)x+b(t)
\end{equation}
where $a(t)$ and $b(t)$ belong to $\mathbb{Z}[t]$.

We denote by $C$ the smooth projective curve with affine equation $x^3+a(t)x+b(t)=0$, which is the smooth compactification of $\E[2]\setminus\{0\}$. We also let $\Delta(t):=-16(4a(t)^3+27b(t)^2)$ be the discriminant of the Weierstrass equation \eqref{QtEq}.

Using an appropriately chosen elliptic curve $E$ of large $\mathbb{Q}(t)$-rank, we shall apply Theorem~\ref{thm:Main} to obtain infinitely many cubic fields with a class group of large $2$-rank.

In \cite{Kih}, Kihara gives an example of an elliptic curve $E$ over $\mathbb{Q}(t)$ of rank at least $14$.  A calculation shows that in Kihara's example, $\E$ contains singular fibers (over $\PP^1_{\mathbb{Q}}$) of types $\mathrm{I}_2$, $\mathrm{I}_4$ (two fibers), and $\mathrm{I}_6$.  The elliptic curve $E$ is obtained as a specialization of a $3$-dimensional family of elliptic curves having rank at least $12$.  By using a different specialization of this family, we may obtain a more advantageous singular fiber configuration (with  Theorem~\ref{thm:Main} in mind) at the expense of (possibly) lowering the  $\mathbb{Q}(t)$-rank.

Specifically, let $E'$ be the elliptic curve over $\mathbb{Q}(p,q,u)$ of rank at least $12$ constructed in \cite{Kih} (in fact, Kihara gives a genus one quartic with $13$ $\mathbb{Q}(p,q,u)$-points; one must choose one of the $13$ points as a base point for the elliptic curve).  Let $E$ be the elliptic curve over $\mathbb{Q}(t)$ obtained from $E'$ by the specialization $p=t$, $q=t+6$, $u=t+1$. Then one can check numerically that there is no prime of universal bad reduction for $E$.

We may write the equation of $E$ in the form \eqref{QtEq}, where the discriminant $\Delta(t)$ is irreducible in $\mathbb{Q}[t]$ of degree $96$.  There is exactly one bad place in $\PP^1_{\mathbb{Q}}$, where $\E$ has a singular fiber of type $I_1$.
 
By specializing Kihara's points, at say $t=1$, and computing the associated height pairing, it is easily verified that $E$ has rank at least $12$ over $\mathbb{Q}(t)$.

Then, according to Theorem~\ref{thm_kummer}, we have
$$
\dim_{\F_2} \Pic(C)[2] \geq \rk_{\Z} E(\Q(t)) \geq 12.
$$

One also checks that there exists a rational value of $t$ for which the equation $x^3+a(t)x+b(t)=0$ has a single real root. By variants of Theorem~\ref{thmclassgp} (resp. Theorem~\ref{thm:Main}) taking into account the places at infinity, it follows that there exists a trigonal morphism $\phi:C\to\PP^1$ such that for all but $O(\sqrt{N})$ integers $t\in\mathbb{Z}$ with $|t|\leq N$, $\mathbb{Q}(P_t)$ is a cubic number field with exactly one real place and
$$
\dim_{\F_2} \Cl(\Q(P_t)) \geq \rk_{\Z} E(\Q(t)) -\rk_{\Z} \OO_{\mathbb{Q}(P_t)}^\times \geq 11.
$$

\begin{remark}
One may possibly go further using famous constructions of Elkies. More precisely, Elkies describes in \cite{Elkies07} constructions of elliptic curves over $\Q(t)$ of ranks $17$ and $18$.  The details of these constructions (e.g., explicit equations for the curves) remain unpublished.  This example is likely to lead to further applications of our techniques.
\end{remark}

%%%%%%%%%%%%%%%%%%%%%%%%%%%%%%%%%%%%%%%%%%%%%%%%%%%%%%%%%%%%%%%%%%%%

\section*{Acknowledgments}

The first author was supported in part by the CIMI Excellence program while visiting the \emph{Centro di Ricerca Matematica Ennio De Giorgi} during the autumn of 2017. The second author was supported in part by NSF grant DMS-1352407.

%%%%%%%%%%%%%%%%%%%%%%%%%%%%%%%%%%%%%%%%%%%%%%%%%%%%%%%%%%%%%%%%%%%%

\providecommand{\bysame}{\leavevmode\hbox to3em{\hrulefill}\thinspace}
\providecommand{\MR}{\relax\ifhmode\unskip\space\fi MR }

%%%%%%%%%%%%%%%%%%%%%%%%%%%%%%%%%%%%%%%%%%%%%%%%%%%%%%%%%%%%%%%%%%%%

\end{document}